\documentclass[12pt]{amsart}
\usepackage[all]{xy}
\usepackage[parfill]{parskip}

\usepackage{verbatim}
\usepackage{color}

\usepackage{amsmath, amscd, graphicx, latexsym, hyperref, times}
\usepackage{color,graphicx}
\usepackage[abs]{overpic}
\usepackage{tikz}
\usepackage{tikz-cd}
\usetikzlibrary{arrows, patterns}

\oddsidemargin .25in \theoremstyle{plain}
\newtheorem{theorem}{Theorem}
\newtheorem{lemma}{Lemma}
\newtheorem{proposition}{Proposition}
\newtheorem{corollary}{Corollary}

\newtheorem{remark}{Remark}

\newtheorem{fact}{Fact}

\numberwithin{equation}{section}
\numberwithin{lemma}{section}
\numberwithin{proposition}{section}
\numberwithin{corollary}{section}
\numberwithin{remark}{section}

\begin{document}

\title{The Obata equation with Robin boundary condition}
\author{Xuezhang Chen}
\address{Department of Mathematics \& IMS, Nanjing University, Nanjing
210093, P. R. China}
\email{xuezhangchen@nju.edu.cn}
\author{Mijia Lai}
\address{School of Mathematical Sciences, Shanghai Jiao Tong University, Shanghai 200240, P. R. China}
\email{laimijia@sjtu.edu.cn}
\author{Fang Wang}
\address{School of Mathematical Sciences, Shanghai Jiao Tong University, Shanghai 200240, P. R. China}
\email{fangwang1984@sjtu.edu.cn}

\thanks{Chen's research was supported by NSFC (No.11771204), A Foundation for the Author of National Excellent Doctoral Dissertation of China (No.201417) and start-up grant of 2016 Deng Feng program B at Nanjing University. Lai's research was supported in part by National Natural Science Foundation of China No.11871331.
 Wang's research was supported in part by National Natural Science Foundation of China No. 11871331 and No. 11571233. }

\date{}

\begin{abstract}
We study the Obata equation with Robin boundary condition $\frac{\partial f}{\partial \nu}+af=0$ on manifolds with boundary,
where $a$ is a non-zero constant. Dirichlet and Neumann boundary conditions were previously studied by Reilly \cite{R},
Escobar \cite{Es} and Xia \cite{X}. Compared with their results, the sign of $a$ plays an important role here. The new
discovery shows besides spherical domains, there are other manifolds for both $a>0$ and $a<0$. We also consider the Obata
equation with non-vanishing Neumann condition $\frac{\partial f}{\partial \nu}=1$.
\end{abstract}
\maketitle

\section{Introduction}\label{sec.intro}
Let $(M,g)$ be a closed manifold with $\mathrm{Ric}_g\geq (n-1)g$, Lichnerowicz~\cite{L} had shown that the first non-zero
eigenvalue $\lambda_1$ of the Laplacian is no less than $n$. Later, Obata~\cite{O} proved $\lambda_1=n$ if and only if
$(M,g)$ is isometric to the standard sphere. This follows from the Obata rigidity theorem characterizing the standard sphere:
a complete manifold $(M,g)$ admits a non-constant function $f$
satisfying the Obata equation
\[
\nabla^2 f+fg=0
\]
if and only if $(M, g)$ is isometric to the standard sphere.

Lichnerowicz type eigenvalue estimate was generalized to manifolds with boundary. Suppose $(M,g)$ is a compact manifold with
boundary and $\mathrm{Ric}_g\geq( n-1)g$. Reilly \cite{R} showed the first eigenvalue $\mu_1$ of the Laplacian
with the Dirichlet boundary condition is no less than $n$, provided that the boundary is mean convex. Moreover,
$\mu_1=n$ if and only if $(M,g)$ is isometric to the standard hemisphere. Escobar~\cite{Es} and Xia~\cite{X} independently
proved that the first eigenvalue $\eta_1$ of the Laplacian with the Neumann boundary condition satisfies $\eta_1\geq n$ if
the boundary is convex. The equality holds if and only if $(M,g)$ is isometric to the standard hemisphere. The analysis of both equalities
reduces to the study of the Obata equation with Dirichlet boundary condition and Neumann boundary condition respectively.

It is thus natural to consider the Obata equation with the Robin boundary condition:
\begin{equation}\label{eq.robin}
\begin{cases}
\nabla^2 f+fg=0 & \textrm{in $M$},
\\
\frac{\partial f}{\partial \nu}+af=0 & \textrm{on $\partial M$},
\end{cases}
\end{equation}
where $\nu$ is the outward unit normal on $\partial M$ and $a$ is a non-zero constant. This is the primary goal of this paper.

Before getting into the complete discussion, let us first introduce some standard models $(M, g)$ where (\ref{eq.robin}) admits
nontrivial solutions. Fix $\theta\in(0,\pi/2)$, let $D^m(\theta)$ be the following domain in $\mathbb{S}^{n}=\{y\in \mathbb{R}^{n+1}: |y|=1\}$:
\begin{equation} \notag
D^m(\theta)=
\begin{cases}
\{ y_1\geq \cos\theta\},
\quad &\textrm{if $m=0$,}
\\
\{ y_{m+2}^2+\cdots +y_{n+1}^2\leq   \sin^2 \theta\},
\quad &\textrm{if $1\leq m\leq n-2$,}
\\
\{ y_{n+1}\leq  \sin \theta\},
\quad &\textrm{if $m=n-1$.}
\end{cases}
\end{equation}

Then $T^m(\theta)=\partial D^m(\theta)$ is either a sphere or a generalized Clifford torus:
\begin{equation} \notag
T^m(\theta)=
\begin{cases}
\{ y_1=\cos\theta\},
\quad &\textrm{if $m=0$,}
\\
\{y_1^2+\cdots+y_{m+1}^2=\cos^2 \theta, \  y_{m+2}^2+\cdots +y_{n+1}^2=  \sin^2 \theta\},
\quad &\textrm{if $1\leq m\leq n-2$,}\\
\{ y_{n+1}=  \sin \theta\},
\quad &\textrm{if $m=n-1$.}
\end{cases}
\end{equation}

Take $M_1=\mathbb{S}^{n}\setminus D^m(\theta)$, $M_2=D^m(\theta)$ and $f=y_{n+1}$.
Then
\[
\begin{cases}
\nabla^2 f+fg=0 & \textrm{in $M_1$},
\\
\frac{\partial f}{\partial \nu_1}+(\cot\theta) f=0 & \textrm{on $\partial M_1$},
\end{cases}
\quad\quad
\begin{cases}
\nabla^2 f+fg=0 & \textrm{in $M_2$},
\\
\frac{\partial f}{\partial \nu_2}-(\cot\theta) f=0 & \textrm{on $\partial M_2$},
\end{cases}
\]
where $\nu_1,\nu_2$ are outward unit normals on $\partial M_1$ and $\partial M_2$ respectively. For
simplicity, we shall use $D^{m}(\theta)$ and $T^{m}(\theta)$ to denote all domains differing by an $SO(n+1)$ action fixing $y_{n+1}$-axis.

With notations as above, we can state the rigidity theorem for $a>0$.

\begin{theorem}[$a>0$]\label{thm.positive}
Let $(M,g)$ be a smooth connected $n$-dimensional compact Riemannian manifold with boundary.
Suppose $f\in C^{\infty}(M)$ is a non-constant function satisfying (\ref{eq.robin}) and denote
$a=\cot \theta>0, \theta\in (0,\pi/2)$. Then $g$ is the spherical metric, i.e., $sec_g\equiv 1$.

Moreover, suppose each boundary component $S$ has constant mean curvature, then $S$ is isometric to
either $T^{m}(\theta)$ ($m\neq n-2$) or a $k$-fold cover of $T^{n-2}(\theta)$. Denote the index $m$ as
$m(S)$.  We have the following more concrete descriptions:
\begin{itemize}
\item[(i)]
If $f|_{\partial M}$ is constant, then $\partial M$ is connected and isometric to $T^{n-1}(\theta)$, and
$(M, g)$ is isometric to $\mathbb{S}^n\setminus D^{n-1}(\theta)$, a geodesic ball of radius $\pi/2-\theta$ in $\mathbb{S}^n$.
(referred to also as a spherical cap)

\item[(ii)]
If $f|_{\partial M}$ is not constant, then $m(S)<n-1$ for any boundary component and one of the following holds:
\begin{itemize}
\item[(a)]
$n\geq 3$, if $m(S)<n-2$ for all boundary components, then
$(M,g)$ is isometric to a spherical domain:
\[
(M,g, f)=\left(\mathbb{S}^n\setminus \bigsqcup_{i=1}^l D^{m_i}(\theta), g_{\mathbb{S}^n}, Ly_{n+1}\right),
\]
where $l$ is the number of boundary components and $m_i<n-2$ for $i=1,...,l$.
\item[(b)]
$n\geq 3$, if there exists a boundary component $S$ such that $m(S)=n-2$, then any other boundary component $S'$
(if exists) must be isometric to $T^{0}(\theta)$ with constant principal curvature $-a$. Gluing small spherical caps
$D^0(\theta)$ along all such $S'$, we obtain a manifold $(\widetilde{M}, \widetilde{g})$, which is a $k$-fold isometric
covering of $\mathbb{S}^n \setminus D^{n-2}(\theta)$. $k$ corresponds to the number of interior maximum points of $f$.
\item[(c)]
$n= 2$, $(M, g)$ is a surface of constant Gaussian curvature $1$ and each boundary curve is of constant geodesic curvature
 $-\cot \theta$ and of length $2\pi k \sin\theta$, where $k=k(S)$ is the number of maximum points of $f$ restricted on $S$.
\end{itemize}
\end{itemize}
\end{theorem}

Ren-Xu~\cite{RX} obtained a Lichnerowicz type estimate for the first non-zero
eigenvalue of the Laplacian with the Robin boundary condition on manifolds with $\mathrm{Ric}_g\geq (n-1)g$.
As an application of Theorem \ref{thm.positive}, we characterize the equality case of their estimates.

\begin{corollary}\label{C1} Let $(M,g)$ be an $n$-dimensional compact manifold with boundary, suppose $\mathrm{Ric}_g\geq (n-1)g$
and the second fundamental form $h\geq -2a g$ ($a>0$) with the mean curvature $H\geq (n-1)/a$. Let $\xi_1$ be the first eigenvalue
of the Laplacian with Robin boundary condition
\begin{align}  \notag
\left\{
  \begin{array}{ll}
    \Delta f+\xi_1 f=0, & \quad\textrm{in $M$}, \\
   \frac{\partial f}{\partial \nu}+af=0, &\quad\textrm{on $\partial M$}.
  \end{array}
\right.
\end{align}
Then $\xi_1\geq n$ and the equality holds if and only if $(M, g)$ is isometric to a spherical cap.
\end{corollary}

For $a<0$, we have the following:
\begin{theorem}[$a<0$]\label{thm.negative}
Let $(M,g)$ be a smooth connected compact $n$-dimensional Riemannian manifold with boundary.
Assume $f\in C^{\infty}(M)$ is a non-constant function satisfying equation (\ref{eq.robin}) and denote
$a=\cot \theta<0, \quad\theta\in (\pi/2,\pi)$.
Set $M_0=\{p\in M: f(p)=0\}$, $\tilde{g}_0=g|_{M_0}$ and $\tilde{\nabla}$ the Levi-Civita connection w.r.t. $\tilde{g}_0$.
We then have following characterization of $(M,g)$:
\begin{itemize}
\item [(i)]
If $f$ is constant on some boundary component, then
$(M, g)$ is isometric to either a geodesic ball of radius ${3\pi}/{2}-\theta$ in $\mathbb{S}^n$ or
the warped product metric
\[
g= dt^2+(\cos t)^2 \tilde{g}_0, t\in[\pi/2-\theta, \theta-\pi/2].
\]
\item[(ii)]
If $f$ is not constant on any boundary component, then $\partial M$ is connected.
Moreover, $M$ is a $\mathbb{Z}_2$-symmetric domain in the warped product space
$$
M_0 \times  [\pi/2-\theta, \theta-\pi/2]_t, \quad g=dt^2+ (\cos t)^2 \tilde{g}_0,
$$
which is bounded by the graphs of $\pm\phi$, where $\phi\in C^{\infty}(\mathring{M}_0)\cap C(M_0)$ satisfies
\begin{align}\notag
\frac{\cos \phi}{\sqrt{1+(\cos\phi)^{-2}|\tilde{\nabla} \phi|_{\tilde{g}_0}^2}} +a\sin\phi=0, \quad\textrm{in $\mathring{M}_0$},
\\ \notag
\textrm{$\phi>0$ in $\mathring{M}_0$ and $\phi=0$ on $\partial M_0$.}
\end{align}
\end{itemize}
\end{theorem}

The idea for analyzing the Obata equation has two-folds: the existence of a non-constant $f$ whose hessian is proportional to the metric
implies a local warped product structure between two regular level sets of the function (c.f. \cite{P}); an isolated critical point (if exist) helps to
determine both the topology and the metric via the gradient flow of $f$.

Working on the manifolds with boundary, the sign of $a$ in the Robin condition plays a role. If $a$ is positive, then there exist both interior maximum and minimum. It follows
that the manifold has spherical metric (constant positive sectional curvature). To determine the boundary, we analyze the second fundamental form carefully.
The discussion leads to an interesting connection with isoparametric hypersurfaces in spheres.

On the other hand, if $a<0$ then the global maximum and minimum may not occur
in the interior. Therefore only warped product structure is to be expected. Nevertheless, the restriction of $f$ to the boundary is a transnormal function, which
controls the geometric and topological structure to some extent.

Another issue complicating the proof is that there might be multiple
boundary components. This can be handled by the behavior of $f$ along its gradient flow lines together with the Robin boundary condition.

We also consider the Obata equation with non-vanishing Neumann boundary condition:
\begin{equation}\label{eq.neumann1}
\begin{cases}
\nabla^2 f+fg=0 &\quad \textrm{in $M$},
\\
\frac{\partial f}{\partial \nu}=1 &\quad \textrm{on $\partial M$}.
\end{cases}
\end{equation}
This serves as a complementary example for the third possibility of the appearance of the maximum and minimum. In this case,
only global minimum is achieved in the interior. This actually simplifies the matter.

\begin{theorem}\label{thm.neumann1}
Let $(M,g)$ be a smooth compact $n$-dimensional Riemannian manifold with boundary and $f\in C^{\infty}(M)$ satisfies equation (\ref{eq.neumann1}).
Then
\begin{itemize}
  \item[(i)] if $f|_{\partial M}=const$, then $(M, g)$ is isometric to a spherical cap;
  \item[(ii)] if $f|_{\partial M}\neq const$, then $(M, g)$ is isometric to the standard hemisphere.
\end{itemize}
\end{theorem}

Finally, we would like to point out that the assumption that each boundary component has constant mean curvature in Theorem~\ref{thm.positive} is for the simplicity of the exposition.
This technical assumption is only needed for Proposition~\ref{P3.2} and its subsequent propositions. Without this assumption, we may run into a lengthy discussion of isoparametric hypersurfaces
with $g=3,4,6$ distinct principal curvatures (the reader is referred to the remark after Proposition~\ref{P3.2} for details). We shall address this issue in a forthcoming paper.
We also remark that one can study generalized Obata equations $\nabla^2 f+h(f)g=0$ (c.f.~\cite{WY}) on manifolds with boundary. Almaraz and Barbosa~\cite{AB} also studied the Obata equation with Robin
boundary condition under another assumption on the mean curvature of the boundary, based on which they obtained Corollary~\ref{C1} as well.

An outline of this paper is as follows:
In Section \ref{sec.gen}, we set up notations and make some general discussion of the manifold which is independent of the sign of $a$.
In Section \ref{sec.pos}, we treat the equation (\ref{eq.robin}) with $a>0$ and prove Theorem \ref{thm.positive}.
In Section \ref{sec.neg}, we treat the equation (\ref{eq.robin})  with $a<0$ and prove Theorem \ref{thm.negative}.
In Section \ref{sec.neu}, we consider equation (\ref{eq.neumann1}) and prove Theorem \ref{thm.neumann1}.
In Appendix \ref{sec.app}, we prove a gluing theorem. This is helpful in our proof for main theorems and might be of independent
interest.

{Acknowledgment: The authors wish to thank Prof. Jianquan Ge, Prof. Guoqiang Wu, Prof. Wenjiao Yan for helpful discussions. The authors thank Prof. Chao Xia for pointing out an inaccuracy in an earlier version of this paper.}

\section{Preliminary}\label{sec.gen}
In this section, we set up notations and prove some preliminary results.
Let $(M^n,g)$ be a smooth connected compact manifold with boundary
and $f$ be a non-constant function satisfying equation (\ref{eq.robin}).

\begin{itemize}
\item[(i)]
\textbf{Ambient Manifold.}
We use $R_{ijkl}$, $\nabla$, $\Delta_g$, etc., for terms associated with $g$ and $:$ for covariant derivative w.r.t. $g$.

\item[(ii)]
\textbf{Boundary.}
We use $\bar{R}_{ijkl}$, $\bar{\nabla}$, $\Delta_{\bar{g}}$, etc., for terms associated with $\bar{g}$, the induced metric on the boundary and
use $;$ for covariant derivative w.r.t. $\bar{g}$ on the boundary.

Denote by $h$ the second fundamental form of $\partial M$ w.r.t. the outward unit normal $\nu$.
Choose an orthonormal frame $\{e_1, ..., e_{n}\}$ with $e_n=\nu$ near the boundary, our convention for the second fundamental form is
\[
h(e_i,e_j)
=\langle \nabla_{e_i}e_n, e_j\rangle_g
=  -\langle \nabla_{e_i}e_j, e_n \rangle_g.
\]

\item[(iii)]
\textbf{Level Sets and Components.}
Let $M_c=\{p\in M: f(p)=c\}$ and $X_{+} (X_{-})$ be a connected component of $\{f>0\} (\{f<0\})$.

\item[(iv)]
\textbf{Boundary Level Sets and Components.}
We denote by $S$, a connected component $\partial M$. Let $S_c=\{p\in S: f(p)=c\}$ and $S_{+} (S_{-})$ be a connected component of $\{f>0\} (\{f<0\})$ in $S$.

\item[(v)]
\textbf{Global Max/Min and Boundary Max/Min.}
Denote $C_{+}=\{ f=\max_{x\in M} f(x)\}$ and $C_{-}=\{ f=\min_{x\in M} f(x)\}$. Let $\bar{f}$ be the restriction of $f$ on a boundary component $S$ and denote $\Sigma_{+}=\{ \bar{f}=\max_{x\in S} \bar{f}(x) \}$,
$\Sigma_{-}=\{\bar{f}=\min_{x\in S} \bar{f}(x)\}$.
\end{itemize}

\begin{lemma}\label{l2.1}
Suppose $f$ satisfies (\ref{eq.robin}), then there exists a constant $L>0$ such that
\[
|\nabla f|^2 +f^2 =L^2.
\]
\end{lemma}

\begin{proof}
Since
\[
\nabla_X (|\nabla f|^2 +f^2)= 2\nabla^2 f(\nabla f, X)+ 2 f \nabla_X f=-2 f \langle \nabla f, X \rangle +2 f\nabla_X f=0,
\]
the conclusion then follows.
\end{proof}

A direct consequence is that $\forall t\neq \pm L$, $M_t$ is a regular hypersurface.
\begin{lemma} \label{geodesic} Let $(N, g)$ be a compact manifold, suppose $f$ is a smooth function such that
its gradient is an eigenvector of its hessian at each point, then the integral curves of $\frac{\nabla f}{|\nabla f|}$ are geodesics.
\end{lemma}

\begin{proof} By assumption, there exists $\varphi$ such that
\[
\nabla^2 f(\nabla f, \cdot)=\varphi \langle \nabla f, \cdot\rangle .
\]
Then for any tangent vector $X$,
\begin{align} \notag
\langle \nabla _{\nabla f} \frac{\nabla f}{|\nabla f|}, X\rangle =\frac{\nabla^2 f(\nabla f, X)}{|\nabla f|}-\langle \nabla f, X\rangle  \frac{\nabla^2f(\nabla f, \nabla f)}{|\nabla f|^3}=0.
\end{align}
The conclusion follows.
\end{proof}

By (\ref{eq.robin}), $\nabla f$ is an eigenvector of $\nabla^2 f$, it follows that the integral curves of $\frac{\nabla f}{|\nabla f|}$ are geodesics.

At this point it is worthwhile to briefly sketch the proof of Obata theorem. As this theme of argument prevails throughout the paper: Gradient flow sets
a canonical diffeomorphism between two regular level sets, for which the Obata equation induces a metric O.D.E. Isolated non-degenerate critical points
prescribe the asymptotic behavior of the metric, which helps to solve the metric O.D.E.

More precisely, let $(M, g)$ be a complete smooth Riemannian manifold, suppose there exists $f$ satisfying $\nabla^2 f+fg=0$,
then Lemma~\ref{l2.1} implies $C_{\pm}$ are the only critical sets of $f$, and if exist, they consist of isolated non-degenerate critical points.
Consider the flow of $\frac{\nabla f}{|\nabla f|}$, the flow lines are all geodesics by Lemma~\ref{geodesic}. The restriction of $f$ on each flow line satisfies $f''+f=0$.
Let $\gamma(s)$ be a flow line such that $\gamma(s)\in M_0$. Then $f(\gamma(s))=L \sin s$. The standard Morse theory implies that $f^{-1}(-L, L)$ is diffeomorphic
to $M_0\times(-\frac{\pi}{2}, \frac{\pi}{2})_s$. Moreover, $g$ can be written as $ds^2+ g(s)$, where $g(s)$ is a family of metrics on $M_0$.

Under this coordinate system, we have
\[
\nabla^2 f=\left(
             \begin{array}{cc}
               f''(s) & 0\\
               0 & f'(s) h(s) \\
             \end{array}
           \right)
           \]
where $h(s)$ is the second fundamental form of the level set $M_0\times \{s\}$. Since $h(s)=\frac{1}{2}g'(s)$, the Obata equation thus becomes a metric O.D.E:
\begin{align}\notag
\frac{1}{2}f'(s) g'(s)+f(s) g(s)=0.
\end{align}
As $s\to \pm \frac{\pi}{2}$, the level sets must converge to two isolated critical points, which are smooth manifold points. From this one can solve for $g$ and prove that $M$ is indeed isometric to the
standard sphere.

Working on the manifolds with boundary, we use above idea to obtain following two useful results.

\begin{proposition} \label{Pwarped} Suppose $\nabla^2 f +fg=0$ on $M$. Take a closed subset $N \subset M_{c}$, denote by $N_{\delta}$ the domain resulted from
the normalized gradient flow of $f$ on $N$ for $t\in[0, \delta]$. Suppose that $N_{\delta}\subset \mathring{M}$, then $g$ is the warped product metric
$dt^2+\frac{\cos^2(\alpha+t)}{\cos^2 \alpha} g_{N}$. Here we set $c=L \sin \alpha$.
\end{proposition}

\begin{proof}
By Lemma~\ref{l2.1} and the assumption, we have $f(N_{\delta})\subset(-L, L)$.

Morse theory implies that $N_{\delta}$ is diffeomorphic to $N\times[0, \delta]_t$. Since we consider the normalized gradient flow, i.e.,
the flow of $\frac{\nabla f}{|\nabla f|}$, the metric $g$ on $N_{\delta}$ takes the form $dt^2 +g(t)$, where $g(t)$ is a family of metrics on $N$.
(\ref{eq.robin}) implies the following metric O.D.E
\begin{align}  \label{metricode}
\frac{1}{2}f'(t) g'(t)+f(t) g(t)=0.
\end{align}
Let $\gamma$ be a flow line of $\frac{\nabla f}{|\nabla f|}$, then $f(\gamma(t))$ satisfies
\begin{align} \label{eq.geodesic}
f''(\gamma(t))+f(\gamma(t))=0,
\end{align}
with $f(\gamma(0))=c$ and $f'(\gamma(0))=\sqrt{L^2-c^2}$.
Since $c=L \sin \alpha$, we solve for (\ref{eq.geodesic}) to get
\[
f(\gamma(t))=L \sin(\alpha+t).
\] Plugging this to (\ref{metricode}) and in view of $g(0)=g_N$, we obtain
\[
g(t)=\frac{\cos^2(\alpha+t)}{\cos^2 \alpha} g_{N}.
\]
\end{proof}

\begin{proposition} \label{Pspherical} Let $(M, g)$ and $f$ be as above. Let $p \in \mathring{M}$ such that $f(p)=L$.
Then for any star shaped region $V \subset T_pM$ containing $p$ where the exponential map $exp_p$ is defined,
then $g|_{exp_p(V)}$ is the spherical metric.
\end{proposition}

\begin{proof}
If $f(p)=L$, by Lemma~\ref{l2.1} and (\ref{eq.robin}), it is a non-degenerate critical point of $f$, thus isolated.
Since all flow lines of $\frac{\nabla f}{|\nabla f|}$ terminates till $f\to L$, it follows that all
geodesics starting at $p$ are backward flow lines. By (\ref{eq.geodesic}), it is easy to see that $f(x)=L \cos d(x, p)$.

Take any star shape region $V$ in $T_pM$, there exists $r_0$ small enough such that $B_{r_0}\subset V$
and $exp_p$ is a diffeomorphism on $B_{r_0}$.

We \textbf{claim} that $g$ on $exp_p(B_{r_0}(p))$ is the spherical metric.

Let $r$ be the radial coordinate associated with $exp_p$. Then the metric has the form $g=dr^2+\bar{g}(r)$, where
$\bar{g}(r)$ is a family of metrics on $\mathbb{S}^{n-1}$.
The metric O.D.E implied by (\ref{eq.robin}) is
\begin{align}  \notag
\frac{1}{2}f'(r) \bar{g}'(r)+f(r) \bar{g}(r)=0.
\end{align}
Also note $f(r)=L \cos r$. Since $p$ is a smooth manifold point, we have
\[
\lim_{r\to 0} \frac{\bar{g}(r)}{r^2}= g_{\mathbb{S}^{n-1}}.
\]
Based on this, we get $\bar{g}(r)=\sin^2 r g_{\mathbb{S}^{n-1}}$. Therefore in a small neighborhood of $p$, $g$ is the spherical metric.

Now for any $q \in  exp_p(V)$, then there exists a geodesic $\gamma: [0,L ]\to exp_p(V)$ connecting $p$ and $q$, with $\gamma(0)=p$.
Consider the part $\gamma|_{s\geq r_0}$, there exists a neighborhood $N_{\delta}$ of the form in Proposition~\ref{Pwarped}
with $N \subset \partial exp_p(B_{r_0})$. By Proposition~\ref{Pwarped} and the above claim, it follows $g$ is the spherical metric on $N_{\delta}$.
\end{proof}

We conclude this section by recalling the notion of the transnormal function and some basic properties of it.
Let $(N, g)$ be a closed Riemannian manifold, a function $f$ of class $C^2$ is called a transnormal function, if
\[
|\nabla f|^2 =b(f),
\]
for some function $b$ of class $C^2$ defined on the range of $f$. Put $V_{+}=\{ f=\max\}$ and $V_{-}=\{ f=\min\}$, they are called focal varieties.
The fundamental property of the focal varities is due to Q.M. Wang~\cite{Wa}.

\begin{theorem}[Wang]
The focal varieties of a transnormal function are smooth submanifolds.
\end{theorem}

In view of Lemma ~\ref{l2.1} and the Robin boundary condition, we have
\begin{align}  \label{e2.1}
|\bar{\nabla} f|^2+(1+a^2)f^2=L^2
\end{align} on $\partial M$. Hence restriction of $f$ on each boundary component is a transnormal function.

Denote by $\Sigma_{+}$ and $\Sigma_{-}$ the focal submanifolds of a boundary component $S$.

\begin{lemma} \label{focal}
If codimensions of $\Sigma_{+}$ and $\Sigma_{-}$ are both greater than $1$, then $\Sigma_{+}$ and $\Sigma_{-}$ are connected.
\end{lemma}

\begin{proof} It follows from the assumption that $S\setminus (\Sigma_{+}\cup \Sigma_{-})$ is connected. In view of (\ref{e2.1}), $S\setminus (\Sigma_{+}\cup \Sigma_{-})$ is
diffeomorphic to $S_0 \times (0,1)$, thus $S_0$ is connected and the conclusion follows.
\end{proof}

\section{$a>0$ and The Proof of Theorem \ref{thm.positive}}\label{sec.pos}

In this section, we assume $f$ is a non-constant function satisfying equation (\ref{eq.robin}) and put
\[
a=\cot \theta>0,\quad\theta\in (0,\pi/2).
\]
We first deal with the simple case that $f$ is constant on a connected boundary component. The idea is
already mentioned in previous section. Namely consider the gradient flow.

\begin{proposition}\label{constantSp}
Suppose there exists a boundary component $S$, such that $f|_{S}$ is constant.
Then $\partial M=S$, which is isometric to $T^{n-1}(\theta)$ and $M$ is isometric
to a geodesic ball of radius $\pi/2-\theta$ in $\mathbb{S}^n$.
\end{proposition}

\begin{proof}
Without loss of generality, we can assume $f|_{S}$ is a positive constant. Hence by (\ref{e2.1})
\[
f|_S= L\sin \theta, \quad \frac{\partial f}{\partial \nu}=-L \cos \theta.
\]
Consider the flow of $S$ along $\frac{\nabla f}{|\nabla f|}$. Clearly, $\frac{\nabla f}{|\nabla f|}$
is opposite to $\nu$ along $S$. Let $\gamma_p(t)$ be the flow line starting from an arbitrary point $p\in S$.
By (\ref{eq.robin}),
\[
f''(\gamma_p(t))+f(\gamma_p(t))=0.
\]
So $f(\gamma_p(t))=L\sin (t+\theta)$. For $t\in(0, \frac{\pi}{2}-\theta)$, $f(\gamma_p(t))> L \sin \theta$, therefore $\gamma_p(t)$ is an interior point.
The only critical point of $f$ is where $f=\pm L$, therefore $M\setminus \{f=\pm L\}$ is diffeomorphic to $S \times_t [0, \frac{\pi}{2}-\theta)$, and
 $g$ takes the form $g=dt^2 +\bar{g}(t)$. The Obata equation implies
\[
\frac{1}{2} \frac{d }{d t}f(\gamma_p(t))\bar{g}'(t)+f(\gamma_p(t))\bar{g}(t)=0.
\]
Since $\bar{g}(0)=\bar{g}|_S$, we obtain
\[
\bar{g}(t)=\frac{\cos^2(t+\theta)}{\cos^2\theta} \bar{g}|_{S}, \quad
g=dt^2 +\frac{\cos^2(t+\theta)}{\cos^2\theta} \bar{g}|_{S}.
\]
Letting $t\rightarrow \pi/2-\theta$, $(S,\bar{g}|_S)$ shrinks to a smooth point, thus
$(S,\bar{g}|_S)=\mathbb{S}^{n-1}(\cos\theta)$ and $(M,g)$ is isometric to a geodesic ball of radius $\pi/2-\theta$ in $\mathbb{S}^n$.
\end{proof}

Next we deal with the general case that $f$ is not constant on any boundary component.
To this end, we shall establish several lemmas and propositions.
The following lemma characterizes inward geodesics from any boundary point in the direction of  $\nabla f$.

\begin{lemma}\label{l3.1} $\forall p \in M$, let $\gamma(t)$ be the geodesic starting at $p$ with $\gamma'(0)=\frac{\nabla f}{|\nabla f|}$, then
\begin{itemize}
  \item if $f(p)> 0$, then $\gamma(t)$ hits an interior maximum before it hits the boundary ;
  \item if $f(p)< 0$, then $\gamma(t)$ hits an interior minimum before it hits the boundary.
\end{itemize}
\end{lemma}

\begin{proof}
The existence of such a geodesic only needs clarification when $p \in \partial M$. In that case, by the Robin condition, if $f(p)>0$ then
\[
0>-af(p)=\frac{\partial f}{\partial \nu}(p)=\nabla f \cdot \nu.
\]
Hence $\nabla f(p)$ points inward. So the geodesic starting at $p$ in the direction of
 $\frac{\nabla f(p)}{|\nabla f(p)|}$ exists at least for a short period of time. By virtue of Lemme~\ref{geodesic}, $\gamma(t)$ coincides with the integral curve of $\frac{\nabla f}{|\nabla f|}$.
Therefore the restriction of $f$ on $\gamma(t)$ satisfies
\[
f''(\gamma(t))+ f(\gamma(t))=0.
\]
Assume $f(p)=L \sin \alpha$ with $\alpha\in(0, \frac{\pi}{2})$, then
\[
f(\gamma(t))=L \sin(\alpha+t).
\]
Suppose $\gamma(t)$ hits $\partial M$ at $t_0$ before it reaches a global maximum, then $t_0<\frac{\pi}{2}-\alpha$ and $f(\gamma(t_0))>0$.  Since
 $\gamma'(t_0)$ forms an acute angle with $\nu$, we have  $\frac{\partial f}{\partial \nu}(\gamma(t_0))>0$. On the other hand, due to the Robin boundary condition
 \[
\frac{\partial f}{\partial \nu}(\gamma(t_0))=-a f(\gamma(t_0))<0,
 \] a contradiction. The proof of second statement is similar.
\end{proof}

\begin{proposition} \label{P3.1}
Let $X_{+}(X_{-})$ be a connected component of $\{f>0\}(\{f<0\})$, then $X_{+}(X_{-})$ can be isometrically embedded into the hemisphere $\mathbb{S}^{n}_{+}$.
\end{proposition}

\begin{proof}
Take a maximum point $p$, i.e., $f(p)=L$. For any unit vector $u\in T_p M$, let $\gamma$ be
the unit-speed geodesic starting at $p$ in the direction $u$. We know all geodesics starting at $p$
coincide with backward flow lines of $\frac{\nabla f}{|\nabla f|}$. It follows that
 $f(\gamma(t))=L \cos t$.

Denote by $T_u$ the time $\gamma(t)$ hits $\partial M$.
Let $\Omega $ be the region in $T_p M$ given under the polar coordinates as
\[
\Omega=\{(t,u)|t<\min\{\frac{\pi}{2}, T_u\}, u\in \mathbb{S}^{n-1}\}.
\]
It is easy to see that the connected component $X_{+}$ containing $p$ is indeed $exp_p(\Omega)$.
By Propostion~\ref{Pspherical}, the conclusion is evident.
Set $p$ as the north pole, then $X_{+}$ is in the upper hemisphere and $\partial X_{+} \cap M_0$
lies on the equator, which is totally geodesic.
\end{proof}

The following proposition is the heart of the proof, which helps to determine $\partial M$. Note (\ref{e2.1})
 implies that $f$ is a transnormal function on $S$. Let $\Sigma_{+}$ and $\Sigma_{-}$ denote its focal submanifolds.
Suppose $m_{+}=\dim \Sigma_{+}$ and $m_{-}=\dim \Sigma_{-}$.

\begin{proposition} \label{P3.2}
Assume each connected boundary component $S$ has constant mean curvature, then $m_{+}=m_{-}=m(S)$ and the second fundamental form of $S$ has at most two
distinct principal curvatures: $-a$ of multiplicity $n-1-m$ and $\frac{1}{a}$ of multiplicity $m(S)$.
\end{proposition}

\begin{proof}
Choose an orthonormal frame $\{e_1, \cdots, e_n\}$ with $e_n=\nu$. In view of (\ref{eq.robin}),
\[
0=f_{:in}=e_ie_n(f)- \nabla_{e_i} e_n (f).
\]
It follows that
\begin{align} \label{e3.1}
-a f_i =h_{ij} f_j.
\end{align}
This means $\bar{\nabla} f$ is a principal direction of $h$ with principal curvature $-a$.

By
\[
\nabla^2 f=\bar{\nabla}^2 f+ h_{ij} f_n,
\]
we have
\begin{align} \label{e3.2}
f_{;ij}+h_{ij}(-af)+ f \bar{g}=0 \quad \text{on $S$}.
\end{align} Hence $\bar{\nabla} f$ is an eigenvector of $\bar{\nabla}^2 f$. By Lemma~\ref{geodesic}, the integral curves of $\frac{\bar{\nabla} f}{|\bar{\nabla} f|}$ are geodesics in $(S, \bar{g})$.

We can further assume that $e_1$ is in the direction of $\bar{\nabla} f$. Hence $f_1\neq 0$ and $f_2=\cdots=f_{n-1}=0$.
Taking covariant derivative to (\ref{e3.1}), we get
\begin{align} \label{e3.3}
-a f_{;ik}=h_{ij;k}f_j+h_{ij} f_{;jk}.
\end{align}
Plugging (\ref{e3.2}) into (\ref{e3.3}), we have
\begin{align} \label{e3.4}
h_{i1;k} f_1+ (h_{ij}+a\delta_{ij})(afh_{jk}-f\delta_{jk})=0.
\end{align}
By Codazzi equation,
\[
h_{i1;k}-h_{ik;1}=R_{ik1n}=0.
\]
The vanishing of Riemannian curvature is due to Proposition~\ref{P3.1} that $g$ is the spherical metric.
Hence (\ref{e3.4}) becomes
\begin{align} \notag
h_{ik;1} f_1+ (h_{ij}+a\delta_{ij})(af h_{jk}- f\delta_{jk})=0.
\end{align}

From this point on, we choose an orthonormal frame $\{e_1, \cdots, e_{n-1}\}$ at a point
$p \in S$ with $f(p)=0$, such that $h_{ij}(p)$ is diagonal. Let $\gamma(s)$ be
the integral curve of $\frac{\bar{\nabla} f}{|\bar{\nabla} f|}$ passing through $p$.
Parallel transport this frame along $\gamma(s)$, i.e., $\bar{\nabla}_{e_1} e_i\equiv 0, 2\leq i\leq n-1$. Hence
\[
h_{ik;1}=e_1( h_{ik})- h(\bar{\nabla}_{e_1} e_i, e_k)-h(e_i, \bar{\nabla}_{e_1} e_k)= \frac{d}{ds} h_{ik}(\gamma(s)).
\]
Then $h_{ik}(\gamma(s))$ reviewed as an $(n-2)\times (n-2)$ matrix valued function satisfies the O.D.E
\begin{align} \label{e3.7}
f_1 \frac{d}{ds} h_{ik}(\gamma(s))  + (h_{ij}+a\delta_{ij})(af h_{jk}- f\delta_{jk})=0.
\end{align}
Hence $h_{ik}$ remains diagonal, i.e, $\{e_2, \cdots, e_{n-1}\}$ are all principal directions.

Thus (\ref{e3.7}) can be viewed as an O.D.E for each principal curvature, which we denote as $\lambda_2, \cdots, \lambda_{n-1}$.
By (\ref{e3.1}, \ref{e3.2}), $f''(\gamma(s))+(1+a^2) f(\gamma(s))=0$ and thus
\[
f(\gamma(s))=\frac{L}{\sqrt{1+a^2}} \sin(\sqrt{1+a^2} s), \quad s\in \left[0, \frac{\pi}{2\sqrt{1+a^2}}\right].
\]
Hence (\ref{e3.7}) reduces to
\begin{align}\label{e3.5}
\sqrt{1+a^2}\cos (\sqrt{1+a^2}s) \lambda_i'(s) + \sin(\sqrt{1+a^2} s) f(s)(\lambda_i(s)+a)(a \lambda_i(s)-1)=0.
\end{align}

According to (\ref{e3.2}), along the tangential direction of $\Sigma_{+}$, $h_{ij}$ has principal
curvature $\frac{1}{a}$ of multiplicity $m_{+}$. The normal directions of $\Sigma_{+}$ are spanned
by $\{\gamma'(s)\}$ with $ s\to \frac{\pi}{2\sqrt{1+a^2}} $, thus all have principal curvature $-a$. Viewing these as boundary values:
\[
\lim_{s\to \frac{\pi}{2\sqrt{1+a^2}}} \lambda_i(s)=-a, \quad \text{or} \quad  \lim_{s\to \frac{\pi}{2\sqrt{1+a^2}}} \lambda_i(s)=\frac{1}{a},
\]
we can solve (\ref{e3.5}) to get
\begin{align} \label{3.6}
 \lambda(s)\equiv -a, \quad \text{or} \quad \lambda(s)=\frac{a\mu \cos(\sqrt{1+a^2}s)+1}{a-\mu\cos(\sqrt{1+a^2}s)}\quad  \text{for some $\mu\in \mathbb{R}$.}
\end{align}

Since $\Sigma_{+}$ has constant mean curvature, it follows $\mu=0$ and consequently principal curvatures of $h_{ij}$ are:
\begin{align} \notag
-a & \quad \text{of multiplicity $n-1-m_{+}$}; \\ \notag
\frac{1}{a} & \quad \text{of multiplicity $m_{+}$}.
\end{align}

A similar argument applies to $S_{-}$. By continuity, it is necessary that $m_{+}=m_{-}$, which is denoted by $m(S)$.
\end{proof}

\begin{remark}
We remark here without assumption of constant mean curvature on each boundary component, we may also get further information regarding the principal curvatures.
It has an interesting connection with isoparametric hypersurfaces in spheres.

We need the following notions: let $X_{+}$ and $S_{+}$ be as above and set
\begin{itemize}
  \item $S_{0}:=\{f=0\}\cap \overline{S_{+}}$;
  \item $X_{0}:=\{f=0\} \cap \overline{X_{+}}$.
\end{itemize}

The corresponding Levi-Civita connections are denoted respectively as
\[
(X_{+}, \nabla), \quad  (S_{+}, \bar{\nabla}), \quad (X_{\delta}, \hat{\nabla}), \quad (S_{\delta}, \tilde{\nabla}).
\]
Choose an orthonormal frame field $e_1,...,e_{n-1}, e_n$ near $S_{+}$ with $e_n=\nu$ and $e_1=\frac{\bar{\nabla}f}{|\bar{\nabla} f|}$, then for $2\leq i,j \leq n-1$ we have
\begin{align} \label{3.7}
\nabla_{e_i} e_j-\bar{\nabla}_{e_i} e_j=-h_{ij} \nu.
\end{align}
and
\begin{align} \label{3.8}
\bar{\nabla}_{e_i} e_j-\tilde{\nabla}_{e_i} e_j=\beta_{ij} \frac{\bar{\nabla} f}{|\bar{\nabla} f|},
 \end{align} where $\beta_{ij}$ is the second fundamental form of $S_{0}$ in $S_{+}$ with respect to $-e_1$.

Clearly $X_{0}$ is in the equatorial sphere of $X_{+}$, thus
\begin{align} \label{3.9}
\nabla_{e_i} e_j-\hat{\nabla}_{e_i} e_j=0.
\end{align}
Restricting (\ref{e3.2}) to $S_{0}$, we find
\[
-\beta_{ij} |\bar{\nabla} f| -af h_{ij}+f g=0,
\]
it follows that $S_{0}$ is totally geodesic in $S_{+}$, i.e., $\beta_{ij}=0$ on $S_{0}$.

Hence combining (\ref{3.7}, \ref{3.8}, \ref{3.9}), we get
\begin{align} \notag
\hat{\nabla}_{e_i} e_j -\tilde{\nabla}_{e_i} e_j= -h_{ij} \nu.
\end{align}
This means the principal curvatures of the second fundamental form of $S_0$ in $X_{0}$ are identical to $h_{ij}$ restricted on $S_{0}$.
In view of (\ref{3.6}), possible principal curvatures are $\{-a, \frac{a\mu+1}{a-\mu}\}$.

Since $X_{0}$ is a region in the standard $S^{n-1}$, we infer that $S_{0}$ is
an isoparametric hypersurface in $S^{n-1}$. Thus according to the fundamental work of M\"{u}nzner~\cite{M}, the number $g$ of distinct principal curvatures could only be $1,2,3,4,6$.
Moreover, the principal curvatures
can be written as
\[
\cot (\theta), \cot (\theta +\frac{\pi}{g}), \cdots, \cot(\theta+\frac{(g-1)\pi}{g}).
\]
It follows by simple calculation that there are only finitely many possible values for $\mu$ and they come in pairs:
\begin{align} \notag
\mu\in \left\{\pm \cot(\frac{\pi}{2}), \pm \cot(\frac{\pi}{3}), \pm\cot(\frac{\pi}{4}), \pm\cot(\frac{\pi}{6})\right\}.
\end{align}

The assumption of each boundary component has constant mean curvature implies that $\mu=0$. Without this assumption, $S_{0}$ could be an isoparametric hypersurface with $g=3,4,6$, which complicates the discussion of
the whole picture of $M$. We shall revisit this issue in a forthcoming paper.

Another condition that rules out the possibility of $g=3,4,6$ is that $a\leq \cot(\frac{\pi}{6})$.
Recall that
 \[
 \lambda(s)=\frac{a\mu \cos(\sqrt{1+a^2}s)+1}{a-\mu\cos(\sqrt{1+a^2}s)}.
 \] If there is some $\mu \in  \{ \pm \cot(\frac{\pi}{3}), \pm\cot(\frac{\pi}{4}), \pm\cot(\frac{\pi}{6})\}$, then the dominator would be zero somewhere. This is impossible.
Thus Proposition~\ref{P3.2} and Theorem~\ref{thm.positive} still hold under the assumption $a\leq \cot(\frac{\pi}{6})$.
\end{remark}

\begin{proposition} \label{P3.3}
Suppose $f$ satisfies (\ref{eq.robin}) and its restriction on any boundary component is not constant. Assume each boundary component has constant mean curvature, we have the following:
\begin{itemize}
  \item if $m(S)=m<n-2$, then $S$ is isometric to $T^m(\theta)$;
  \item if $m(S)=n-2$, then $S$ is isometric to $\mathbb{S}^{n-2}(\cos \theta) \times \mathbb{S}^1(k \sin \theta)$.
  Here $k$ corresponds to the number of connected components of $S_{+}$ (same as the number of connected components of $S_{-}$)
\end{itemize}
\end{proposition}

\begin{proof}
Take a connected component $X_{+}$ and consider $[X_{+}]^2$, the double of $X_{+}$ across $\partial X_{+}\cap M_0$.
Write $X=[X_{+}]^2$ for simplicity. In view of Proposition~\ref{P3.1}, we could identify the unique maximum of $f$
in $X_{+}$ with the north pole, then the metric double can be viewed across the equator, which is totally geodesic.
Moreover $\partial S_{+}$ is also totally geodesic in $S_{+}$. This can be seen from (\ref{e3.2}) as $f=0$ on
$\partial S_{+}$. Hence $X$ is a smooth spherical domain with smooth boundary. Moreover, $\partial X$ is obtained
by doubling of some $S_{+}$'s. We also make an odd extension of $f$ to $X$ across $\partial X_{+}\cap M_0$, it is
evident that $f=L y_{n+1}$.  

Proposition~\ref{P3.2} implies that $\partial X$ are isoparametric hypersurfaces with at most two distinct principal curvatures.
By the theory of isoparametric hypersurfaces in $\mathbb{S}^n$ (c.f. \cite{Ca}), it follows that each connected component
of $\partial X$ is isometric to a generalized Clifford torus
\[
T^m(\theta)=\mathbb{S}^{m}(\cos\theta)\times \mathbb{S}^{n-1-m}(\sin\theta), \quad 0\leq m\leq n-2.
\](Here by $T^0(\theta)$ we mean one copy of $\mathbb{S}^{n-1}(\sin \theta)$)

Therefore $S_{+}$ is isometric to $T^m_{+}(\theta):=\{(y_1, \cdots, y_{n+1})\in T^m(\theta) | y_{n+1}>0\}$.
Similarly, $S_{-}$ is isometric to $T^m_{-}(\theta)$.

If $m<n-2$, the focal submanifolds $\Sigma_{\pm}$ of $S$ have codimension at least $2$.
By Lemma~\ref{focal}, we have $S\setminus S_0=S_{+}\cup S_{-}$. Hence $S$ is isometric to $T^m(\theta)$.

If $m=n-2$, $S_{+}$ is isometric to $T^{n-2}_{+}$, which is $\mathbb{S}^{n-2}(\cos \theta)\times (0, \pi \sin \theta)$.
Hence on $S$, each $S_{\pm}$ is boarded with one $S_{\mp}$ on $\mathbb{S}^{n-2}\times\{0\}$ and
with another $S_{\mp}$ on $\mathbb{S}^{n-2}\times \{\pi\sin \theta\}$. It follows that $S$ is isometric to
a $k$-fold covering of $T^{n-2}(\theta)$, where the covering factor through the $\mathbb{S}^1$ factor.
Accordingly, $k$ is the number of the connected components $S_{+}$, as well as the number of the connected components $S_{-}$.
\end{proof}

Finally, we are in position to prove
\begin{proposition}\label{nonconstSp}
Suppose $f$ satisfies (\ref{eq.robin}) and its restriction on any boundary component is not constant. Assume each boundary component has constant mean curvature, then one of the followings holds.
\begin{itemize}
\item[(i)]
$n\geq 3$, if $m(S)<n-2$ for each connected boundary component $S$,
then
\[
(M,g, f)=\left(\mathbb{S}^n\setminus \bigsqcup_{i=1}^k D^{m_i}(\theta), g_{\mathbb{S}^n}, Ly_{n+1}\right).
\]
Here $D^{m_i}(\theta)$ are disjoint domains of model types up to a rotation around $y_{n+1}$-axis, and
 $l$ represents the number of boundary components.

\item[(ii)]
$n\geq 3$, if there exists some connected boundary component $S$ such that $m(S)=n-2$,
then any other connected boundary component $S'$ (if exists) must be isometric to
$\mathbb{S}^{n-1}(\sin\theta)$ with constant principal curvature $-a$. After gluing
geodesic balls of radius $\theta$ along all such $S'$'s, we obtain a manifold with boundary
$(\widetilde{M}, \widetilde{g})$ which is a $k$-fold isometric covering of
$\mathbb{S}^n \setminus D^{n-2}(\theta)\simeq \mathbb{B}^{n-1}\times \mathbb{S}^1$.

\item[(iii)]
$n=2$, $(M, g)$ is a surface of constant Gaussian curvature $1$ with boundary of
constant curvature $-a$. Moreover, each boundary component $S$ is a circle of
length of $2\pi k \sin\theta$, where $k$ is the number of connected components of $S_+$.
\end{itemize}
\end{proposition}

\begin{proof}
(i) If $m(S)<n-2$, by Proposition~\ref{P3.3} $S$ is isometric to $T^{m}(\theta)$.
According to Theorem~\ref{gluing} in the appendix, gluing $(D^m(\theta), Ly_{n+1})$ to $(M, f)$
along $(S,f|_{S})$ will eliminate this boundary component without changing the equation.
We can do a similar gluing to all $S$ and thus obtain a connected closed manifold
$(\widetilde{M}, \widetilde{g})$ and a smooth non-constant function $\widetilde{f}$ such that
\[
\widetilde{\nabla}^2\widetilde{f} +f\widetilde{g}=0 \quad \textrm{in $\widetilde{M}$}.
\]
By Obata theorem, $(\widetilde{M}, \widetilde{g})$ is the standard round sphere $\mathbb{S}^n$.

(ii) Let $S$ be a boundary component with $m(S)=n-2$.
First, consider the simple case:
\begin{itemize}
\item[($\ast$)]
there is no boundary component $S'$ such that $m(S')=0$.
\end{itemize}

Take a connected component $S_{+}\subset S$, let $X_{+}$ be the connected component containing $S_{+}$.

We first \textbf{claim} there exists no other $S'_{+}\subset X_{+}$.

Since $S_{+}$ is isometric to $T^{n-2}_{+}(\theta)=\mathbb{S}^{n-2}(\cos\theta) \times \mathbb{S}^{1}_{+}(\sin \theta)$,
we infer that (up to a rotation around $y_{n+1}$ axis)
\[
\partial X_{+}\cap M_0\subset \{ y_n\geq \sin \theta, y_{n+1}=0 \}\sqcup \{y_n\leq -\sin \theta, y_{n+1}=0\}.
\]
If there were other $S'_{+}\subset X_{+}$, then $\partial S'_{+}$ must be
contained in $\{ y_n\geq \sin \theta, y_{n+1}=0\}\sqcup \{y_n\leq -\sin \theta, y_{n+1}=0\}$ as well.

Notice $S'_{+}$ is isometric to $T^{m'}_{+}(\theta)$, where $m'=m(S'_{+})$. Then we have
\[
\partial S'_{+}= \mathbb{S}^{m'}(\cos \theta) \times \mathbb{S}^{n-2-m'}(\sin \theta).
\]
If $1\leq m'<n-2$, $\partial S'_{+}$ a connected antipodal symmetric hypersurface
in $\partial X_{+}$, thus cannot be contained in the disjoint union
$\{ y_n\geq \sin \theta, y_{n+1}=0\}\sqcup \{y_n\leq -\sin \theta, y_{n+1}=0\}$.
If $m'=n-2$, then $\partial S'_{+}$ must coincide with $\partial S_{+}$.  In any case, we have proved the claim.

We next \textbf{claim} that $\partial M$ is connected.

Otherwise, there is another boundary component $S'$ such that $0<m(S')\leq n-2$, $S\cap S'=\emptyset$ and $S'$ is the closest one to $S$.
Let $\gamma(t)$ be a geodesic realising the distance between $S$ and $S'$ in $M$. Without loss of generality, we may assume $\gamma(t)$ starts at $p\in S_+$.
Obviously $\gamma'(0)\perp S_+$. Let $X_+$ be the connected component which contains $S_+$. The previous claim prevents  appearance of any other $S'$ in $X_{+}$.
Note $X_{+}$ can be isometrically embedded in $\mathbb{S}^{n}_{+}$. By a simple computation, $\gamma(t)$ will hit $S_+$ again
in $X_+$ before it escapes out of $X_{+}$ to meet with $S'$.

Therefore, $\partial M=S$, which is isometric to $\mathbb{S}^{n-2}(\cos \theta)\times \mathbb{S}^{1}( k\sin\theta)$.
By previous two claims and the proof of Proposition~\ref{P3.3}, we deduce that each $X_{\pm}$ is isometric to
\[
\mathbb{S}^{n}_{\pm} \setminus D^{n-2}_{\pm}(\theta).
\]

Notice $\partial X_{\pm}=\{ y_n\geq \sin \theta, y_{n+1}=0\}\sqcup \{y_n\leq -\sin \theta, y_{n+1}=0\}$.
Therefore each $X_{\pm}$ is boarded with one $X_{\mp}$ on $\{ y_n\geq \sin \theta, y_{n+1}=0\}$ and with another
$X_{\mp}$ on $\{y_n\leq -\sin \theta, y_{n+1}=0\}$. Hence $M$ is isometric to $k$-fold covering of $\mathbb{S}^{n}\setminus D^{n-2}(\theta)$.

Now for the general case. For any $S'$ such that $m(S')=0$, using the same idea as in the proof of (i),
we can glue a geodesic ball $(D^0(\theta), Ly_{n+1})$ along $(S',f|_{S'})$ to eliminate this boundary component without changing the equation.
Then after gluing along all such $S'$, we get a manifold $(\widetilde{M}, \widetilde{g})$
and a smooth non-constant function $\widetilde{f}$ such that ($\ast$) is satisfied. By ($\ast$), $\partial \widetilde{M}=S$
is connected and $(\widetilde{M}, \widetilde{g})$ is an isometric covering of $\mathbb{S}^n\setminus D^{n-2}(\theta)$ of degree $k$.
$k$ corresponds to the number of connected components $X_{+}$. ($X_{-}$)

(iii) When $n=2$ and $f|_{\partial M}$ is not constant, then for each connected boundary component
$S$ is just a circle, which is the flow of $\bar{\nabla} f/|\bar{\nabla} f|$. Let $\sigma(t)$ be
the integral curve of $\bar{\nabla} f/|\bar{\nabla} f|$ starting from a boundary maximum point $q\in S$.
Then
\[
f(\sigma(t))=\sin\theta\cos(t/\sin\theta).
\]
The non-negative and non-positive parts of $f(\sigma(t))$ correspond to the connected components of $S_+$ and $S_-$.
If $S_+$ has $k$ components, then the length of $\sigma(t)$ is $2\pi k\sin\theta$.
\end{proof}

\begin{remark}
For $n=2$ and $f|_{\partial M}$ is non-constant, the global picture of $M$ can be more complicated.
For each $S$, $S_+$ may have different numbers of connected components. Moreover, if pick a connected component $X_+$ of $M_+$,
then $\partial X_+\cap M_0$ may have several connected components which are not identical.
\end{remark}

\textbf{Proof of Theorem \ref{thm.positive}}
The proof is a combination of Proposition~\ref{constantSp} and Proposition~\ref{nonconstSp}.

 \hfill $\Box$

\textbf{Proof of Corollary \ref{C1}} Suppose $\xi_1=n$, and let $f$ be an eigenfunction corresponding to $n$, i.e.,
\begin{align}   \label{3.4}
\left\{
  \begin{array}{ll}
    \Delta f+n f=0, & \quad\textrm{in $M$}, \\
   \frac{\partial f}{\partial \nu}+af=0, &\quad\textrm{on $\partial M$}.
  \end{array}
\right.
\end{align}

Applying Reilly's formula to $f$, we have
\begin{align} \label{3.3}
& \int_{M}[ (\Delta f)^2-|\nabla^2 f|^2-\mathrm{Ric} (\nabla f, \nabla f)]dV
\\ \notag
=&\ \int_{\partial M} \left[2 \bar{\Delta} f \left( \frac{\partial f}{\partial \nu}\right)+H \left( \frac{\partial f}{\partial \nu}\right)^2+h(\bar{\nabla} f, \bar{\nabla}f)
\right]dA.
\end{align}

Using $|\nabla^2f|^2\geq (\Delta f)^2/n$ and $\mathrm{Ric}\geq (n-1)g$, we find that
\begin{align} \notag
\mathrm{l.h.s ~of~ }(\ref{3.3})\leq (n-1)\int_{M} (nf^2-|\nabla f|^2) dV.
\end{align}
Multiplying the first equation of (\ref{3.4}) by $f$ and integrating on $M$, we obtain
\[
\int_{M} (nf^2-|\nabla f|^2) dV=a\int_{\partial M} f^2 dA,
\]
and consequently
\[
\mathrm{l.h.s ~of ~}(\ref{3.3})\leq a(n-1) \int_{\partial M} f^2 dA.
\]
On the other hand, recalling that $h\geq -2a g$, $H\geq (n-1)/a$ and $ \frac{\partial f}{\partial \nu}+af=0$, we obtain
\[
\mathrm{r.h.s~ of ~}(\ref{3.3})\geq a(n-1) \int_{\partial M} f^2 dA.
\]
Therefore $|\nabla^2 f|^2=(\Delta f)^2/n$, which implies that $f$ satisfies the Obata equation
(\ref{eq.robin}) with Robin boundary condition. Also note equality implies that $H\equiv \frac{n-1}{a}$. Now a routine check through all cases in Theorem~\ref{thm.positive}
implies that the spherical cap is the only one compatible with the given boundary condition. \hfill $\Box$

\section{$a<0$ and The Proof of Theorem \ref{thm.negative}}\label{sec.neg}
\label{sec.neg}
In this section, we assume $f$ is a non-constant function satisfying equation (\ref{eq.robin}) and
set $a=\cot\theta<0, \quad\theta\in (\pi/2, \pi)$.

\begin{proposition}\label{constantSn1}
$f$ has an interior critical point if and only if $\partial M$ has only one component and $f|_{\partial M}$ is constant.
In this case, $(M, g)$ is a geodesic ball of radius $\frac{3\pi}{2}-\theta$ in $\mathbb{S}^n$.
\end{proposition}

\begin{proof}
Without loss of generality, we assume $q$ is an interior critical point of $f$, and $f(q)=L$. Let $\gamma_0(t)$ be a geodesic
realising the distance between $q$ and $\partial M$, then the time $t_0$ that $\gamma_0$ meets $\partial M$ must
satisfies $t_0<\pi$. Otherwise, the exponential map $exp_q$ maps
$B_{\pi}\subset T_{q}M$ into $\mathring{M}$. By Proposition~\ref{Pspherical}, $exp_q(B_{\pi})$ is isometric to the standard sphere, a contradiction.

Note that $f(\gamma_0(t))=L \cos t$, then we have at $t_0$ that
\begin{align} \notag
\frac{d}{dt}(f(\gamma_0(t)))|_{t_0}
=\frac{\partial f}{\partial \nu}|_{\gamma_0(t_0)}=-af(\gamma_0(t_0)).
\end{align}
So $ -L\sin t_0 =-a L\cos t_0$, i.e. $t_0=\frac{3\pi}{2}-\theta$.

Proposition~\ref{Pspherical} implies that $g$ is the standard spherical metric on the geodesic ball $B_{ t_0}(q)$.
Hence, if $\gamma_1'(0)$ is a sufficiently small perturbation of $\gamma_0'(0)$, then $\gamma_1(t)$
intersects with $\partial M$ transversely at some time $t_1$ satisfying $t_0\leq t_1< \pi$.
However, we have at $t_1$ that

\begin{align} \notag
\frac{d}{dt}(f(\gamma_1(t)))|_{t_1}
\leq \frac{\partial f}{\partial \nu}|_{\gamma_1(t_1)}=-af(\gamma_1(t_1))<0.
\end{align}
Thus $-L\sin t_1 \leq -a L\cos t_1$, which implies $t_1\leq\frac{3\pi}{2}-\theta$.

Therefore $t_1=t_0$ and $\gamma_1(t)$ also realizes the distance between $q$ and $\partial M$. Let $T\subset \mathbb{S}T_qM$
be the set of directions such that the geodesic starting from $q$ first meets $\partial M$ at $t=t_0$.
Above argument shows that $T$ is open. $T$ is also closed by the regularity of the metric.
Hence, $\mathring{M}$ coincides with the geodesic ball $B_{t_0}(q)$, which is a geodesic ball of radius $\frac{3\pi}{2}-\theta$ in $\mathbb{S}^n$.

Conversely, if $\partial M$ has only one connected component and $f|_{\partial M}$ is constant,
we may assume on $\partial M$
\[
f=-L \sin \theta, \quad
\frac{\partial f}{\partial \nu}=L \cos \theta<0.
\]
For any $p\in \partial M$, consider the gradient flow $\gamma_p(t)$ of $\nabla f/|\nabla f|$ starting from $p$. Then
$f(\gamma_p(t))$ satisfies
\[
f''(\gamma_p(t))+f(\gamma_p(t))=0, \quad f(\gamma_p(0))=-L \sin \theta, \quad f'(\gamma_p(0))=-L \cos \theta>0.
\]

Hence, $f(\gamma_p(t))=-L\sin (t+\theta)$.
Since $\partial M$ is connected, all flow lines will continue till they reach the critical set $C_{+}$. Therefore $M\setminus C_{+}$ is diffeomorphic to
$\partial M \times _t (0, \frac{3\pi}{2}-\theta)$. It follows that $C_{+}$ is a singleton, say $p$. Then $\mathring{M}=exp_p(B_{\frac{3\pi}{2}-\theta})$
and the conclusion thus follows.
\end{proof}

\begin{proposition}\label{constantSn2}
Suppose $S\subset \partial M$ is a connected component such that $f|_S$ is constant and $\partial M$ contains at least two components. Then
$(M,g)$ is isometric to the warped product
\[
 g=dr^2+\frac{(\cos r)^2}{(\cos(\theta-\frac{\pi}{2}))^2} \bar{g}_{S}, \quad
r\in [\pi/2-\theta, \theta-{\pi}/{2}].
\]\end{proposition}

\begin{proof}
Without loss of generality, we assume
\[
f=-L\sin \theta, \quad
\frac{\partial f}{\partial \nu}=L \cos \theta<0, \quad\textrm{on $S$}.
\]
For any $p\in S$, consider the gradient flow $\gamma(t)$ of $\nabla f/|\nabla f|$ starting from $p$. Then
$f(\gamma_p(t))=-L\sin (t+\theta)$.

Since $\partial M$ contains at least $2$ components, by Proposition~\ref{constantSn1},
there is no interior maximum/minimum point. Thus any geodesic $\gamma(t)$ starting from
$p\in S$ will meet $\partial M \setminus S$ at a time $t_p<\frac{3\pi}{2}-\theta$. Let $t_0=\min_{p\in S} t_p$ and $\gamma_0$ be the
geodesic realizing $t_0$.
Then $\gamma_{0}(t)$ meets a boundary component $S'\neq S$ perpendicularly at $t_0$ and
\begin{align}\notag
\frac{d}{dt}f(\gamma_{p_0}(t))|_{t_0}
=\frac{\partial f}{\partial \nu}|_{\gamma_0(t_0)}=-af(\gamma_0(t_0)).
\end{align}
Thus
$ t_0=2\pi-2\theta$.

It follows $f^{-1}([-L \sin \theta, L \sin \theta])$ is isometric to $S\times [0,t_0]_t$ with the warped product
\begin{equation}\label{eq.4.1}
g=dt^2+\frac{(\sin(t-\theta))^2}{(\sin\theta)^2} \bar{g}_{S}, \quad t\in [0,t_0].
\end{equation}
Since the metric in a small neighborhood of $\gamma_0(t)$ is given by the warped product
type (\ref{eq.4.1}), for $p_1$ close enough to $p_0$ in $S$, the geodesic $\gamma_{1}(t)$
starting from $p_1$, will also intersect with $S'$ transversely at some $t_0\leq t_1<\frac{3\pi}{2}-\theta $. Thus
\begin{align} \notag
\frac{d}{dt}(f(\gamma_1(t)))|_{t_1}
\geq \frac{\partial f}{\partial \nu}|_{\gamma_1(t_1)}=-af(\gamma_1(t_1))\geq 0.
\end{align}
Consequently,
$ t_1\leq 2\pi-2\theta =t_0$.
This implies $t_1=t_0$. By running a similar open and closed argument as in the
proof of Proposition~\ref{constantSn1}, we know that all gradient flow lines starting from $S$ will meet $S'$ at the same time $t_0$.
The conclusion thus follows.
\end{proof}

The remaining case is that there is a connected boundary component $S$ such that $f|_S$ is not constant.
Due to Proposition~\ref{constantSn1}, there is no interior critical point. The global maximum and minimum are achieved on the boundary.

Recall the notations. For $c\in[-L\sin \theta, L\sin \theta]$, we have
$$
\begin{gathered}
M_c=\{p\in M: f(p)=c\}, \quad
M_{\pm}=\{p\in  M: \pm f(p) \geq 0\};
\\
S_c=\{p\in S: f(p)=c\}, \quad
S_{\pm}=\{p\in \partial  M: \pm f(p) \geq 0\}.
\end{gathered}
$$

\begin{lemma}\label{connect2}
\begin{itemize}
\item[(1)]
For any $p\in M$ with $f(p)> 0$, the backward flow line of $\nabla f/|\nabla f|$ starting at $p$ will
meet $M_0$ before it reaches $\partial M$. For any $p\in M,  f(p)< 0$, the forward flow line of $\nabla f/|\nabla f|$ starting at $p$ will meet $M_0$ before it reaches $\partial M$.

\item[(2)]
For $0< c_1<c_2$, the backward flow of  $\nabla f/|\nabla f|$  defines an injective map from $M_{c_2}$ to $M_{c_1}$.
For $0> c_1>c_2$, the forward flow of  $\nabla f/|\nabla f|$  defines an injective map from $M_{c_2}$ to $M_{c_1}$.
\end{itemize}
\end{lemma}

\begin{proof}
Clearly (2) follows from (1).

Take $p\in \mathring{M}$ with $f(p)>0$, we set $f(p)=L\sin\phi$ for some $\phi \in(\pi-\theta, \pi)$.
Let $\gamma(t)$ be the flow line of $-\nabla f/|\nabla f|$ starting at $p$, then $f(\gamma(t))= L\sin (\phi+t)$.
If $\gamma(t)$ meets the boundary before it reaches $M_0$, say at $t_0$, then $f(\gamma(t_0))>0$ and $t_0<\pi-\phi$.  We have
\[
\nabla f \cdot \nu=\frac{\partial f}{\partial \nu}  |_{t=t_0} = -a f(\gamma(t_0)) >0.
\]
However, $\gamma'(t_0)=-\nabla f/|\nabla f|$ forms an acute angle with $\nu$, a contradiction.

If $p\in \partial M, f(p)>0$, then $\nabla f/|\nabla f|$ points outward at $p$.
Hence, when going backward a bit, $\gamma(t)$ will become an interior point. By above argument, it will meet $M_0$ before $\partial M$.
\end{proof}

By Lemma \ref{connect2}, for $c=\pm L\sin\phi, \phi\in(\pi-\theta, \pi)$,
the backward/forward gradient flow of $\nabla f/|\nabla f|$ defines a map
\[
\Psi_{c}^{0}: M_c\longrightarrow M_0.
\]
Obviously, $\Psi_c^0$ is an injective map.
Therefore, we can view $M$ as a bounded domain in the warped product space $(\widetilde{M}, g)$,
where
\begin{equation}\label{warped}
g=d\phi^2 +  (\cos\phi)^2 \tilde{g}_0 \quad \phi\in [\theta-\pi, \pi-\theta]_{\phi}.
\end{equation}

\begin{proposition}\label{graph}
As a domain of the warped product space (\ref{warped}), $\partial M$ is characterised as follows:
\begin{itemize}
\item[(i)]
The backward/forward flow of  $\nabla f/|\nabla f|$
defines a diffeomorphism by
$$
\Psi_{\pm}: S_{\pm}  \longrightarrow  M_0
$$
and $\Psi_{\pm}|_{\partial S_{\pm}}=Id$.

\item[(ii)] For $p\in M$, denote it as $p=(x,\phi)$ where $x\in M_0$. Then
\[
\Psi_{+}^{-1}(x)=(x,  \phi_{+}(x)),\quad  \Psi_{-}^{-1}(x)=(x,  -\phi_{-}(x)),
\]
where $\phi_{\pm}\in C^{\infty}(\mathring{M}_0)\cap C(M_0)$ are positive functions satisfying
$$
\frac{\cos \phi}{\sqrt{1+(\cos\phi)^{-2}|\tilde{\nabla}\phi|^2}} +a\sin\phi=0
$$
with $\phi|_{\partial M_0}=0$.
Here $\tilde{\nabla}$ is the covariant derivative on $(M_0, \tilde{g}_0)$.
\end{itemize}
\end{proposition}

\begin{proof}
First, consider the interiors of $S_{\pm}$ and $M_0$.
Any $p\in S_{\pm}$ with $f(p)=c\neq 0$ lies on the boundary of  $M_c$. Then $\Psi_{\pm}(p)=\Psi_c^0(p)$.
The regularity of $\Psi_{\pm}$ in the interior of $\partial_{\pm} M$ follows from the regularity of
the boundary and the isometry of $\Psi_c^0$. Similarly,  in the interior of $M_0$ the smoothness
of $\phi_{\pm}$ comes from the smoothness of its graph.
For any $p=(x, \phi_{+}(x))\in S_{+}$ where $x\in M_0$, the outward unit normal is
\[
\nu=\frac{(-\nabla \phi_{+}, 1)}{|(\nabla \phi_{+}, 1)|_g} = \frac{(-\nabla \phi_{+}, 1)}{\sqrt{1+(\cos\phi_{+})^{-2} |\tilde{\nabla}\phi_{+}|^2_{g_0}}}.
\]
Under the coordinates $(x,\phi)$ in the warped product space (\ref{warped}), $f(x,\phi)=L\sin\phi$.
Hence,
\[
\frac{\partial f}{\partial \nu} = \nabla f \cdot \nu= (0,L\cos\phi )\cdot \nu |_{\phi=\phi_{+}(x)} = \frac{L\cos\phi_{+}}{\sqrt{1+(\cos\phi_{+})^{-2} |\tilde{\nabla}\phi_{+}|^2_{g_0}}}.
\]
Since $\frac{\partial f}{\partial \nu} +af=0$ on $\partial M$, we have the equation for $\phi_+$. The proof for $\phi_-$ is the same.

Next, consider the extension of $\phi_{\pm}$ and $\Psi_{\pm}$ to $\partial M_0$.

For a sequence $\{p_m\}\subset \mathring{M}_0$,
we show that $\phi_{\pm}(p_m)\rightarrow 0$ if and only if $d(p_m, \partial M_0)\rightarrow 0$ as $m\rightarrow +\infty$
and hence $\phi_{\pm}$ can be extended continuously to $\partial M_0$ such that $\phi_{\pm}|_{\partial M_0}=0$.
If $\phi_{\pm}(p_m)\rightarrow 0$ and $d(p_m, \partial M_0)$ does not converge to $0$, then there exist $\epsilon_0>0$ and $p\in \mathring{M}_0$ and a subsequence $p_{m_i}$ such that
$$
p_{m_i}\rightarrow p, \quad d(p, \partial M_0)\geq \epsilon_0.
$$
So $p$ is an interior point and $\phi_{\pm}$ is continuous at $p$ by above discussion. Hence, $\phi_{\pm}(p)=0$. This contradicts that $p$ is an interior point.
Conversely, if $d(p_m, \partial M_0)\rightarrow 0$, then for any $c>0$, $p_m$ will be outside of $N_{\pm, c}$ for $m\geq m_c$. Hence, $\phi_{\pm}(p_m)<c$ for $m\geq m_c$. Since $c$ is arbitrary, we have $\phi_{\pm}(p_m)\rightarrow 0$.

Similarly, for $p\in\partial M_0$, $\{p_m\}$ converges to $p$ if and only if $\Psi^{-1}_{\pm}(p_m)$ converges to $p$.
Hence, $\Psi_{\pm}$ can be extended continuously to $\partial_{\pm}M$ such that
$\Psi_{\pm}|_{\partial M_0}=Id. $
To prove it, suppose $\mathrm{M}_0\ni p_m\rightarrow p\in \partial M_0$, by (i), $\phi_{\pm}(p_m)\rightarrow 0$. Since
\[
\mathrm{dist} (p_m, \Psi^{-1}_{\pm}(p_m))\leq \phi_{\pm}(p_m),
\]
we have $\Psi^{-1}_{\pm}(p_m)\rightarrow p$. Conversely, if $\Psi^{-1}_{\pm}(p_m)$ converges to $p\in \partial M_0$, then  $\phi_{\pm}(p_m)\rightarrow 0$. Hence, $p_m\rightarrow p$.
\end{proof}

The next lemma shows $\phi_+=\phi_-$.

\begin{lemma}\label{unique}
There exists a unique solution to the following equation
$$
\begin{gathered}
\frac{\cos \phi}{\sqrt{1+(\cos\phi)^{-2}|\tilde{\nabla}\phi|^2}} +a\sin\phi=0, \quad \mathrm{~~in~~} M_0,
\end{gathered}
$$
such that $0<\phi<\pi-\theta$ in $M_0$ and  $\phi=0$ on $\partial M_0$.
\end{lemma}

\begin{proof}
Suppose $\phi_+\not\equiv\phi_-$. WLOG, we assume
$$
\max_{M_0}|\phi_+-\phi_-| = \phi_+(p)-\phi_-(p)>0 \Longrightarrow
0<\phi_-(p)<\phi_+(p)\leq \pi-\theta<\pi/2.
$$
Then by a direct computation,
$$
\begin{aligned}
&\ \tilde{\nabla}(\phi_+-\phi_-)\cdot \tilde{\nabla}(\phi_++\phi_-)
\\
=&\ \cos^4\phi_+\left(\frac{1}{a^2\sin^2\phi_+} -\frac{1}{\cos^2\phi_+} \right)
-\cos^4\phi_-\left(\frac{1}{a^2\sin^2\phi_-} -\frac{1}{\cos^2\phi_-} \right).
\end{aligned}
$$
At $p$, we get $LHS=0$ but $RHS<0$. This indicates that $\phi_+=\phi_-$.
\end{proof}

\textbf{Proof of Theorem \ref{thm.negative}}
The proof follows from Propositions~\ref{constantSn1}-~\ref{graph} and Lemma~\ref{unique}.\hfill $\Box$

\begin{remark}
By Proposition~\ref{graph}, $\phi=\phi_{+}=\phi_-$ is a transnormal function on $M_0$. That indeed poses some fibration structure on $M_0$. But in general, it is still impossible
to completely determine $M_0$.
\end{remark}

\section{Discussion of equation (\ref{eq.neumann1}) and The Proof of Theorem \ref{thm.neumann1}}\label{sec.neu}
In this section, we shall follow a similar strategy as the discussion of equation (\ref{eq.robin}) to treat (\ref{eq.neumann1}), which we recall here
\[
\begin{cases}
\nabla^2 f+fg=0 & \quad\mathrm{in~~} M,
\\
\frac{\partial f}{\partial \nu}=1 & \quad\mathrm{on~~}\partial M.
\end{cases}
\]

Indeed, if $f$ is constant on the boundary, then the flow along the gradient vector field implies that the manifold is isometric to a geodesic ball in $\mathbb{S}^{n}$. If $f$ is not constant on the boundary, we show that there exists a unique interior minimum point and $M$ can be isometrically embedded in $\mathbb{S}^{n}$, from which the geometry of $M$ follows.

We first recall several facts whose analogs have been already explored in Section~\ref{sec.gen}.

\begin{fact}\label{fact.1}
There exists a constant $L>1$ such that
\[
|\nabla f|^2+f^2=L^2 \quad\mathrm{in~~} M.
\]
\end{fact}

\begin{fact} \label{fact.2}
On each connected boundary component $S$,
\[
|\bar{\nabla} f|^2+f^2= L^2-1.\]
Hence $f$ is a transnormal function on each boundary component.
In particular, let $\Sigma_{\pm}=\{f=\pm\sqrt{L^2-1}\}$ be the focal submanifolds with dimension $m_{\pm}$.
\end{fact}

\begin{fact}\label{fact.3}
The integral curves of the gradient vector fields ${\nabla f}/{|\nabla f|}$
and ${\bar{\nabla} f}/{|\bar{\nabla} f|}$ are geodesics w.r.t. $(M, g)$ and $(\partial M, \bar{g})$ respectively.
\end{fact}

Next we prove a lemma, which manifests the distinct feature of the non-vanishing Neumann boundary condition.
Recall
\[
C_{\pm}=\{q\in M: f(q)=\pm L\}.
\]

\begin{lemma} \label{lem5.1}
 $C_+=\emptyset$ and $C_-$ consists of a single point $q_{-}\in \mathring{M}$.
\end{lemma}
\begin{proof}
We prove $C_+=\emptyset$ first.
Suppose not, then there exists $p\in \mathring{M}$ such that $f(p)=L$. Let $\gamma: [0, L]\to M$ be a unit-speed geodesic realizing the distance from $p$ to $\partial M$.

We \textbf{claim} $L \geq \pi$.

Indeed, we have
\[
f(\gamma(t))= L \cos t.
\]
Since $\gamma'(L)$ is perpendicular to $\partial M$, thus
\[
-L \sin L=f'(\gamma(L))=\nabla f \cdot \gamma'(L)=\frac{\partial f}{\partial \nu}=1.
\]
A contradiction is arrived if $L <\pi$. However, if $L \geq \pi$, then by
Proposition~\ref{Pspherical}, $g$ is the spherical metric on $exp_p(B_p(\pi))$, which yields a closed sphere, a contradiction again.
In sum, $C_+=\emptyset$.

Next we prove that $C_{-}$ is a singleton. Take $p \in \mathring{M}\setminus C_{-}$, let $\gamma(t)$ be the integral curve of $-\frac{\nabla f}{|\nabla f|}$ starting at $p$.

We \textbf{claim} that $\gamma(t)$ hits $C_{-}$ before it hits $\partial M$.

The reason is similar. Set $f(p)=L \cos \alpha$ ($\alpha\in(0,\pi)$), then
\[
f''(\gamma(t))+f(\gamma(t))=0,
\]
from which we infer $f(\gamma(t))=L \cos (\alpha+t)$. If $\gamma(t)$ hits $\partial M$ first, it must hit at some $T<\pi-\alpha$. Clearly $\gamma'(T)$ forms an acute angle with $\nu$,
thus
\[
0< \gamma'(T) \cdot \nu=-\frac{1}{|\nabla f|} \frac{\partial f}{\partial \nu} =-\frac{1}{|\nabla f|}.
\]
A contradiction.

Suppose $C_-=\{q_{-1}, \cdots, q_{-m}\}$.  Let
\[
N_i:=\left\{
\begin{array}{cl}
 p\in \mathring{M} : & \text{the geodesic starting from $p$ in the direction} \\
   &\text{  of $-\frac{\nabla f}{|\nabla f|}$ passes through $q_{-i}$ first}
 \end{array}\right\}.
\]

Above argument shows that $\mathring{M}$ can be written as disjoint union of $N_i's$. Next we \textbf{claim} $N_i$ is open for each $i$.

Take $p\in N_i$, then there exists a geodesic segment $\{\gamma(t):0\leq t\leq t_0\}$ starting from $p$ in the direction of $-{\nabla f}/{|\nabla f|}$ passing through $q_i$ first.
By Proposition~\ref{Pspherical} there exists a neighborhood $U$ of $\{\gamma(t):0\leq t\leq t_0\}$ such that $g|_{U}$ is the standard spherical metric.
It follows that there exists a neighborhood $W\subset U$ of $p$, such that $\forall~ p'\in W$, the geodesic starting from $p'$ in the direction of $-{\nabla f}/{|\nabla f|}$ will also pass through $q_i$ first.

Thus we have $\mathring{M}=\bigsqcup_{i} N_i$ as a disjoint union of open sets. Thus there must be only one of such $N_i$,
This means $C_-$ consists of a single point, say $q_-$.
\end{proof}

\begin{proposition} \label{P5.1}
$(M, g)$ can be isometrically embedded in the standard sphere $\mathbb{S}^{n}$.
\end{proposition}

\begin{proof}
Denote by $q_-$ the unique global minimum point of $f$. For the unit speed geodesic $\gamma_u(t)$ starting from $q_-$ in the direction $u$,
we know $f(\gamma_u(t))=-L \cos t$.
By Lemma~\ref{lem5.1} and Fact~\ref{fact.2}, $\max_{M}f=\max_{\partial M}f=\sqrt{L^2-1}$.
Hence $\gamma(t)$ hits the boundary at a time $T_u<\pi$. Let $\Omega$ be the star shaped region given under the polar coordinates as
\[
\Omega=\{ (r,u)| r<T_u, u\in\mathbb{S}^{n-1}\}\subset T_{q_{-}} M.
\]
Then $\mathring{M}=exp_{q_{-}}(\Omega)$. By Proposition~\ref{Pspherical}, $g$ is the spherical metric. Identify $q_{-}$ with the south pole,
we have an obvious isometric embedding of $M$ into $\mathbb{S}^{n}$ as a star shaped region with respect to $q_{-}$.
\end{proof}

\textbf{Proof of Theorem~\ref{thm.neumann1}}
We analyze $\partial M$ in a similar way as in Proposition~\ref{P3.3}.
Choose an orthonormal frame near $\partial M$, such that $e_n=\nu$, $e_1=\frac{\bar{\nabla} f}{|\bar{\nabla} f|}$.
In view of (\ref{eq.neumann1}), for $i\neq n$
\begin{align} \label{e5.1}
0=f_{:in}=\nabla_{e_i} \nabla_{e_n} f-\nabla_{\nabla_{e_i}e_n} f= -h_{ij} f_j.
\end{align}
It follows $\bar{\nabla} f$ is a principal direction with principal curvature $0$.
The restriction of (\ref{eq.neumann1}) on $\partial M$ implies
\begin{align} \label{e5.2}
\bar{\nabla}^2 f +h_{ij}+ fg=0.
\end{align}
Covariant differentiation in (\ref{e5.1}) yields
\[
h_{ij} f_{jk}+ h_{ij;k} f_j=0.
\]
Plugging (\ref{e5.2}) in, we have
\begin{align} \notag
h_{i1;j} f_1+ h_{ij}(-h_{jk}-f \delta_{jk})=0.
\end{align}

Since $g$ is the standard spherical metric, we have $h_{i1;j}=h_{ij;1}$ by Codazzi equation. Thus
\begin{align} \label{e5.3}
h_{ij;1} f_1+ h_{ij}(-h_{jk}-f \delta_{jk})=0.
\end{align}

Let $\gamma(s)$ be an integral curve of $\frac{\bar{\nabla} f}{|\bar{\nabla} f|}$ starting at $p\in \partial M$ with $f(p)=0$. By (\ref{e5.2}), we find that
$f''(\gamma(s))+f(\gamma(s))=0$.
It follows that
\[
f(\gamma(s))=\sqrt{L^2-1} \sin (s), \quad s\in[-\frac{\pi}{2}, \frac{\pi}{2}].
\]

Choose an orthonormal basis $e_2, \cdots, e_{n-1}$ at $p$ corresponding to the principal directions of $h_{ij}$, then parallel transport them along $\gamma(s)$.
Since
\[
h_{ij;1}=e_1(h_{ij})- h(\bar{\nabla}_{e_1} e_i, e_j)- h(e_i, \bar{\nabla}_{e_1} e_j)=\frac{d}{ds} h_{ij}(\gamma(s)),
\]
$\{e_2, \cdots, e_{n-1}\}$ remain principal directions of $h$. We may assume their corresponding principal curvatures are $\lambda_2(s), \cdots, \lambda_{n-1}(s)$, and
consequently (\ref{e5.3}) becomes an O.D.E for each $\lambda_i(s)$ as
\[
\sqrt{L^2-1} \cos (s) \lambda_i'(s) -\sqrt{L^2-1} \sin(s) \lambda_i (s)=\lambda_i^2(s).
\]

Observe that
\[
(\sqrt{L^2-1} \lambda_i(s)\cos(s))'=\lambda_i^2(s)\geq 0,
\]
it follows that $\lambda_i(s) \cos (s)$ is non decreasing. However, $\lim_{s\to \pm \frac{\pi}{2}} \lambda_i(s) \cos (s)=0$, it follows that $\lambda_i(s)\equiv 0$, $\forall i$.
Hence $\partial M$ is totally geodesic and $M$ must be isometric to the hemisphere.

\hfill $\Box$

\begin{remark}
We have previously made a mistake on analyzing the second fundamental form of the boundary. When re-examing the proof of Escobar~\cite{Es} on Obata equation with Neumann boundary condition,
we found a similar mistake. Above argument (as well as that in Proposition~\ref{P3.2} of considering the second fundamental form along an integral curve) is
inspired by C.Y. Xia's proof of Obata equation with Newmann boundary condition.  Escobar's proof can be accordingly revised.
\end{remark}

\section{Appendix: Gluing}\label{sec.app}
In this section, we prove a gluing theorem for $(M, g, f)$ where $(M, g)$ is a smooth Riemannian manifold with boundary and $f\in C^{\infty}(M)$ satisfies the Obata equation.

Let  $r$ be the distance function to $\partial M$.
Then we can identify a collar neighborhood of $\partial M$ with $ \partial M\times [0,\epsilon)_r$ via the normal exponential map,
such that $g$ takes the form
$$
g=dr^2+\bar{g}(r).
$$
where $\bar{g}$ is a smooth family of metrics on $\partial M$.
Denote the Taylor expansions of  $g$ and $f$ in variable $r$ as follows:
\begin{equation}\label{taylor}
\begin{gathered}
 \bar{g}(r)= \sum_{k=0}^{\infty} r^k \bar{g}_{k} + O(r^{\infty}), \quad
 f=\sum_{k=0}^{\infty} r^k\bar{f}_k + O(r^{\infty}),
\end{gathered}
\end{equation}
where $\bar{g}_k\in C^{\infty}(\partial M, S^2T^*\partial M), \bar{f}_k\in C^{\infty}(\partial M)$ for all $k=0,1,2,\cdots$.
In particular $\bar{g}_0$ is the induced metric on $\partial M$; $\bar{f}_0$ and $ -\bar{f}_1$ are the Dirichlet and Neumann data respectively.

\begin{lemma}\label{asymp}
 Suppose there exists a non-constant function $f\in C^{\infty}(M)$ satisfying
$$
\begin{cases}
\nabla^2 f+\varphi(f)g=0 &\quad \mathrm{in~~} M,
\\
f|_{\partial M} =\bar{f}_0 &\quad \mathrm{on~~} \partial M,
\\
-\frac{\partial f}{\partial \nu}=\bar{f}_1\neq 0 & \quad\mathrm{on~~} \partial M,
\end{cases}
$$
where  $\varphi\in C^{\infty}(\mathbb{R})$ and $\nu$ is the outward unit normal on $\partial M$. Then the Taylor expansions (\ref{taylor}) are determined by $(\bar{g}_0, \bar{f}_0, \bar{f}_1)$.
\end{lemma}
\begin{proof}
Now $\bar{g}_0, \bar{f}_0, \bar{f}_1$ are given.
Choose a local coordinate patch  $(U: x^1,..., x^{n-1})$ for $\partial M$. Then near the boundary, $(U\times [0,\epsilon): x^1,..., x^{n-1}, x^n=r)$ is a local coordinate patch for $M$.
We use  $\alpha, \beta, \gamma,\delta,\cdots$ to denote indices $\in\{1,...,n-1\}$. Direct computation shows that the Christoffel symbols for $g$ are as follows:
$$
\begin{aligned}
&\Gamma_{nn}^n=\Gamma_{nn}^{\alpha}=\Gamma_{n\alpha}^{n}=0,
\\
&\Gamma_{\alpha\beta}^n = -\frac{1}{2}\bar{g}'_{\alpha\beta}(r)
=-\frac{1}{2}\sum_{k=0}^{\infty} (k+1) r^k [\bar{g}_{k+1}]_{\alpha\beta}+ O(r^{\infty}),
\\
&\Gamma_{\alpha n}^{\beta} = \frac{1}{2}\bar{g}^{\beta\gamma}(r) \bar{g}'_{\alpha\gamma}(r)
=\frac{1}{2}\bar{g}^{\beta\gamma}(r)\sum_{k=0}^{\infty} (k+1) r^k [\bar{g}_{k+1}]_{\alpha\gamma}+ O(r^{\infty}),
\\
&\Gamma_{\alpha\beta}^{\gamma} = \frac{1}{2} \bar{g}^{\gamma\delta}(r) \Lambda_{\alpha\beta\delta}(\bar{g}(r))
= \frac{1}{2} \bar{g}^{\gamma\delta}(r)  \sum_{k=0}^{\infty} r^k \Lambda_{\alpha\beta\delta}(\bar{g}_k)+ O(r^{\infty}),
\end{aligned}
$$
where $\Lambda_{\alpha\beta\delta}(g)=\partial_{\alpha}g_{\beta\delta}+ \partial_{\beta}g_{\alpha\delta}-\partial_{\delta}g_{\alpha\beta}$. Notice here
$$
\bar{g}^{\alpha\beta}(r) = \sum_{k=0}^{\infty} r^k G^{\alpha\beta}_k +O(r^{\infty}),
$$
where $G^{\alpha\beta}_k\in C^{\infty}(\partial M, S^2T_*\partial M)$ is determined by $\bar{g}_0, ..., \bar{g}_k$. In particular,
$$
G^{\alpha\beta}_0=[\bar{g}_0]^{\alpha\beta}, \quad
G^{\alpha\beta}_1=-[\bar{g}_1]^{\alpha \beta},\quad
G_2^{\alpha\beta} =-[\bar{g}_2]^{\alpha\beta}+[\bar{g}_1]^{\alpha\gamma}[\bar{g}_1]_{\gamma}^{\beta}.
$$
Here the lifting of indices is w.r.t. metric $\bar{g}_0$. Similarly,
$$
\varphi(f)=\sum_{k=0}^{\infty} r^kF_k +O(r^{\infty}),
$$
where $F_k\in C^{\infty}(\partial M)$ is determined by $\bar{f}_0, \bar{f}_1, \cdots, \bar{f}_k$ and $\varphi(\bar{f}_0), \varphi'(\bar{f}_0), \cdots, \varphi^{(k)}(\bar{f}_0)$. In particular
$$
F_0=\varphi(\bar{f}_0), \quad F_1= \varphi'(\bar{f}_0)f_1, \quad
F_2=\frac{1}{2}\varphi''(\bar{f}_0)\bar{f}_1^2+ \varphi'(\bar{f}_0)\bar{f}_2.
$$
The equations are divided into several groups according to various directions:
\begin{itemize}
\item[(i)]
$\nabla_n\nabla_n f+\varphi(f)g_{nn}=\partial_r^2 f +\varphi(f)=0$ implies the ODE for $f$ in variable $r$:
$$
f'' +\varphi(f)=0, \quad f(0)=f_0, \quad f'(0)=f_1.
$$
Since $\varphi$ is smooth, the ODE has a unique solution determined by initial data $f_0, f_1$.
Hence, the Taylor expansion of $f$ is also determined by $f_0, f_1$.

\item[(ii)]
$\nabla_{\alpha}\nabla_{\beta} f+\varphi(f)g_{\alpha\beta}
=\partial_{\alpha}\partial_{\beta} f -\Gamma_{\alpha\beta}^{n}\partial_rf-\Gamma_{\alpha\beta}^{\gamma}\partial_{\gamma}f +\varphi(f)g_{\alpha\beta}=0$ implies that the $k$-th coefficient in the Taylor expansion  of the left hand side vanishes for all $k$.
Notice that
the coefficient of $r^k$ in $-\Gamma_{\alpha\beta}^{n}\partial_rf$ is
$$
\begin{gathered}
\frac{1}{2}\sum_{j=1}^{k+1} j(k+2-j) f_{k+2-j}[\bar{g}_j]_{\alpha\beta}
\\
= \frac{1}{2}\sum_{j=1}^{k} j(k+2-j) f_{k+2-j}[\bar{g}_j]_{\alpha\beta} + \frac{k+1}{2}f_1[\bar{g}_{k+1}]_{\alpha\beta};
\end{gathered}
$$
and the coefficient of $r^k$ in $\partial_{\alpha}\partial_{\beta} f -\Gamma_{\alpha\beta}^{\delta}\partial_{\delta}f +\varphi(f)g_{\alpha\beta}$ is given by
$$
\partial_{\alpha}\partial_{\beta}\bar{f}_k -\frac{1}{2}\sum_{i=0}^{k}\sum_{j=0}^{k-i} G_i^{\delta\gamma}\Lambda_{\alpha\beta\gamma}(\bar{g}_j)\partial_{\delta}\bar{f}_{k-i-j} + \sum_{i=0}^k F_i [\bar{g}_{k-i}]_{\alpha\beta}.
$$
Therefore,
$$
\begin{aligned}
- \frac{k+1}{2}\bar{f}_1[\bar{g}_{k+1}]_{\alpha\beta}
=&\ \partial_{\alpha}\partial_{\beta}\bar{f}_k
+ \frac{1}{2}\sum_{j=1}^{k} j(k+2-j) \bar{f}_{k+2-j}[\bar{g}_j]_{\alpha\beta}
\\
&\
-\frac{1}{2}\sum_{i=0}^{k}\sum_{j=0}^{k-i} G_i^{\delta\gamma}\Lambda_{\alpha\beta\gamma}(\bar{g}_j)\partial_{\delta}\bar{f}_{k-i-j}
+ \sum_{i=0}^k F_i [\bar{g}_{k-i}]_{\alpha\beta}.
\end{aligned}
$$
The right hand side only involves $\bar{f}_0,..., \bar{f}_k$, $\bar{g}_0,..., \bar{g}_k$ and $\varphi(\bar{f}_0), \varphi'(\bar{f}_0), ..., \varphi^{(k)}(\bar{f}_0)$.
Given $\varphi$, if $\bar{f}_1\neq 0$, then $\bar{g}_{k+1}$ is determined by $\bar{f}_0,..., \bar{f}_{k+1}, \bar{g}_0, ..., \bar{g}_{k}$.

\item[(iii)]
$\nabla_{\alpha}\nabla_{n} f+\varphi(f)g_{\alpha n}=\partial_{\alpha}\partial_rf-\Gamma^{\beta}_{\alpha n} \partial_{\beta}f =0$ implies the coefficient of $r^k$ vanishes:
$$
\begin{gathered}
(k+1)\partial_{\alpha}\bar{f}_{k+1}-\frac{1}{2} \sum_{i=0}^k \sum_{j=0}^{k-i} (i+1)[\bar{g}_{i+1}]_{\alpha\delta}G_j^{\beta\delta}\partial_{\beta}\bar{f}_{k-i-j}=0.
\end{gathered}
$$
This impose further restriction on the asymptotics.
\end{itemize}
\end{proof}

\begin{remark}
In above proof, (ii) actually means on each level set $\partial M_{r}=\{r=constant\}$ with induced metric $g(r)$, $f|_{\partial M_r}$ satisfies
$$
\bar{\nabla}^2 f+\varphi(f) g(r)-h(r)f'(r)=0,
$$
where $h(r)=\frac{1}{2}g'(r)$ is the second fundamental form and $\bar{\nabla}$ is the connection w.r.t. $g(r)$;
(iii) actually means on $\partial M_{r}$,
$$
\bar{\nabla}_{\alpha}f'(r)-[h(r)]_{\alpha\beta} \bar{\nabla}^{\beta}f=0.
$$
\end{remark}

\begin{theorem}\label{gluing}
For $i=1,2,3$,
let $(M_i, g_i) $ be smooth compact manifolds with common boundary, i.e.
\[
(\partial M_1, g_1|_{\partial M_1})=(\partial M_2, g_2|_{\partial M_2})=(\partial M_3,g_3|_{\partial M_3}),
\]
and let $f_i\in C^{\infty}(M_i)$.
Given $\varphi\in C^{\infty}(\mathbb{R})$.
we further assume $f_1$ and $f_2$ satisfy
$$
\nabla^2 f_1+\varphi(f_1)g_1=0,\quad \textrm{in $M_1$};\quad
\nabla^2 f_2+\varphi(f_2)g_2=0,\quad \textrm{in $M_2$}
$$
with common Dirichlet and Neumann data:
\[
 f_1|_{\partial M_1}=f_2|_{\partial M_2},
\quad \frac{\partial f_1}{\partial \nu_1}=\frac{\partial f_2}{\partial \nu_2}\neq 0.
\]
Then $(M_1, g_1, f_1)$ can be smoothly glued to $(M_3, g_3, f_3)$ across the boundary if and only if $(M_2, g_2, f_2)$ can be smoothly glued to $(M_3, g_3, f_3)$.
\end{theorem}

\begin{proof}
While gluing two manifolds across the common boundary, we simply take the collar neighborhoods of the boundary in each manifold and  geodesic normal defining functions and glue them naturally along the common boundary.
Then the regularity of glued metric and function can be read from the asymptotical expansions at the boundary.
By Lemma \ref{asymp},  $(M_1,g_1, f_1)$ and $(M_2,g_2, f_2)$ has the same asymptotical expansions. So if one can be smoothly glued to $(M_3, g_3, f_3)$, then the same to the other.
\end{proof}

\begin{remark}
In both Lemma \ref{asymp} and Theorem \ref{gluing}, we only need the Neumann data to be nonzero on a dense subset of the boundary.
\end{remark}

\end{document}